\documentclass{amsart}
\usepackage{amsmath}
\usepackage{amssymb}
\usepackage{amsthm}
\usepackage{enumerate}
\usepackage[pdftex]{graphicx}
\usepackage{caption}
\theoremstyle{definition}
\newtheorem{definition}{Definition}[section]
\theoremstyle{plain}
\newtheorem{lemma}[definition]{Lemma}
\newtheorem{theorem}[definition]{Theorem}

\newtheorem{proposition}[definition]{Proposition}

\theoremstyle{remark}

\makeatletter
\@namedef{subjclassname@2020}{
  \textup{2020} Mathematics Subject Classification}
\makeatother

\newcommand{\mycl}{\operatorname{cl}}

\begin{document}
\title[Morse theory in definably complete d-minimal structures]{Morse theory in definably complete d-minimal structures}
\author[M. Fujita]{Masato Fujita}
\address{Department of Liberal Arts,
Japan Coast Guard Academy,
5-1 Wakaba-cho, Kure, Hiroshima 737-8512, Japan}
\email{fujita.masato.p34@kyoto-u.jp}
\author[T. Kawakami]{Tomohiro Kawakami}
\address{Department of Mathematics,
	Wakayama University,
	Wakayama, 640-8510, Japan}
\email{kawa0726@gmail.com}

\begin{abstract}
Consider a  definable complete d-minimal expansion $(F, <, +, \cdot, 0, 1, \dots,)$ of an oredered field $F$. 
Let $X$ be a definably compact definably normal definable $C^r$ manifold and $2 \le r <\infty$.
We prove that the set of definable Morse functions is open and dense 
in the set of definable $C^r$ functions on $X$
with respect to the definable $C^2$ topology.
\end{abstract}



\maketitle

\section{Introduction}\label{sec:intro}

Let $\mathcal M=(F, +, \cdot, <, \dots)$ be a definably complete d-minimal expansion of 
an ordered field $F$.
Everything is considered in $\mathcal M$,
the term ``definable" is used in the sense of ``definable with
parameters in $\mathcal M$",
each definable map is assumed to be continuous and $2 \le r<\infty$.


Definable $C^r$ Morse functions in an o-minimal expansion of 
the standard structure of a real closed field are considered in \cite{PS}, \cite{K}.

In this paper we consider a definable $C^r$ version of Morse theory in definably complete d-minimal structures
when $2 \le r<\infty$.

Definable $C^r$ manifolds are studied in \cite{PS}, \cite{BO} \cite{K}.

Theorem 1.1 (Theorem 2.14) and 1.3  are our main results.

\begin{theorem}\label{thm:Morse}
	Consider a d-minimal expansion of an ordered field $\mathcal F=(F,<,+,\cdot, 0,1,\ldots)$.
	Let $M$ be a definably compact $\mathcal D^2$ submanifold of $F^n$.
	The set of all definable Morse functions on $M$ is open and dense in the set $\mathcal D^2(M)$ of $\mathcal{D}^2$.functions of $M$
\end{theorem}

Theorem 1.1 is a generalization of \cite{K}.

\begin{theorem}[2.5 \cite{FK2}]
Consider a d-minimal expansion of an ordered field $\mathcal F=(F,<,+,\cdot, 0,1,\ldots)$.
	Let $\mathcal F=(F,<,+,\cdot,0,1,\ldots)$ be a definably complete d-minimal expansion of an ordered field.
	Every definably normal $\mathcal D^r$ manifold is definably imbeddable into some $F^n$, and its image is a $\mathcal D^r$  submanifold of $F^n$.
\end{theorem}


%
%
%

Combining Theorem 1.2 and 1.3, we have the following theorem.


\begin{theorem}
Consider a d-minimal expansion of an ordered field $\mathcal F=(F,<,+,\cdot, 0,1,\ldots)$.
Let $X$ be a definably compact definably normal $\mathcal{D}^r$ manifold.
Then the set of definable Morse functions is open and dense in the set $D(X)$ of $\mathcal{D}^r$ functions on $X$
with respect to the definable $C^2$ topology.
\end{theorem}

\section{Preliminary and Proof of Theorem 1.1}\label{sec:preliminary}
Recall the definitions of d-minimality and definably completeness.

\begin{definition}
An expansion of a dense linear order without endpoints $\mathcal F=(F,<,\ldots)$ is \textit{d-minimal} 
if for every $m$ and definable subset $A$ of $F^{m+1}$,
there exists an $N \in \mathbb{N}$ such that for every $x \in F^m$
the set $\{y \in F|(x, y) \in A\}$ has non-empty interior or a union of at most $N$ discrete sets (\cite{F}, \cite{MT}).

%
The expansion $\mathcal F$ is \textit{definably complete} if any definable subset $X$ of $F$ has the supremum and  infimum in $F \cup \{\pm \infty\}$.
(\cite{M}).

\end{definition}

The definition of dimension is found in \cite[Definition 3.1]{Fuji4}.

\begin{definition}[Dimension]\label{def:dim}
Consider an expansion of a densely linearly order without endpoints $\mathcal F=(F,<,\ldots)$.
Let $X$ be a nonempty definable subset of $F^n$.
The dimension of $X$ is the maximal nonnegative integer $d$ such that $\pi(X)$ has a nonempty interior for some coordinate projection $\pi:F^n \rightarrow F^d$.
We consider that $F^0$ is a singleton with the trivial topology.
We set $\dim(X)=-\infty$ when $X$ is an empty set.
\end{definition}

\begin{definition}[\cite{FK}]
	Consider a definably complete expansion $\mathcal F=(F,<,+, \cdot, 0,1,\ldots)$ of an ordered field.
	Suppose that $1 \le r < \infty$.
	
	(1)
	A pair $(M, \{\varphi_i:U_i \rightarrow U'_i\}_{i \in I})$ of a topological space and a finite family of homeomorphisms is a \textit{definable $\mathcal C^r$ manifold} or a \textit{$\mathcal D^r$ manifold} if  
	\begin{itemize}
		\item $\{U_i\}_{i \in I}$ is a finite open cover of $M$, 
		\item $U'_i$ is a $\mathcal D^r$ submanifold of $F^{m_i}$ for any $i \in I$ and,
		\item the composition $(\varphi_j|_{U_i \cap U_j}) \circ (\varphi_i|_{U_i \cap U_j})^{-1}:\varphi_i(U_i \cap U_j) \rightarrow \varphi_j(U_i \cap U_j)$ is a $\mathcal D^r$ diffeomorphism whenever $U_i \cap U_j \neq \emptyset$.
	\end{itemize}
	Here, the notation $\varphi_i|_{U_i \cap U_j}$ denotes the restriction of $\varphi_i$ to ${U_i \cap U_j}$.
	The family $\{\varphi_i:U_i \rightarrow U_i'\}_{i \in I}$ is called a \textit{ $\mathcal D^r$ atlas} on $M$.
	We often write $M$ instead of $(M, \{\varphi_i:U_i \rightarrow U'_i\}_{i \in I})$ for short.
	Note that a $\mathcal D^r$ submanifold is naturally a $\mathcal D^r$ manifold.

	In the o-minimal setting, a $\mathcal D^r$ manifold is defined as the object obtained by pasting finitely many definable open sets.
	$\mathcal D^r$ submanifolds are pasted in our definition. 
	If we adopt the same definition of $\mathcal D^r$ manifolds as in the o-minimal setting, $\mathcal D^r$ manifolds of dimension zero is a finite set because $F^0$ is a singleton.
	A $\mathcal D^r$ submanifold of dimension zero is not necessarily a $\mathcal D^r$ manifold in this definition.
	It seems to be strange, so we employed our definition of $\mathcal D^r$ manifolds.
	
	Consider a $\mathcal D^r$ manifold $M$, two $\mathcal D^r$ atlases $\{\varphi_i:U_i \rightarrow U_i'\}_{i \in I}$ and $\{\psi_j:V_j \rightarrow V_j'\}_{j \in J}$ on $M$.
	They are \textit{equivalent} if, for all $i \in I$ and $j \in J$,
	\begin{itemize}
		\item the images $\varphi_i(U_i \cap V_j)$ and $\psi_j(U_i \cap V_j)$ are definable  open subsets of $U'_i$ and $V'_j$, respectively, and
		\item the $\mathcal D^r$ diffeomorphism $(\psi_j|_{U_i \cap V_j}) \circ (\varphi_j|_{U_i \cap V_j})^{-1}:\varphi_i(U_i \cap V_j) \rightarrow \psi_j(U_i \cap V_j)$ are definable whenever $U_i \cap U_j \neq \emptyset$.
	\end{itemize}
	The above relation is obviously an equivalence relation.
	
	A subset $X$ of the $\mathcal D^r$ manifold $M$ is \textit{definable} if $\varphi_i(X \cap U_i)$ are definable for all 
	$i \in I$. 
	When two atlases $\{\varphi_i:U_i \rightarrow U_i'\}_{i \in I}$ and $\{\psi_j:V_j \rightarrow V_j'\}_{j \in J}$ of a $\mathcal D^r$ manifold $M$ is equivalent, it is obvious that a subset of the $\mathcal D^r$ manifold $(S,\{\varphi_i\}_{i \in I})$ is definable if and only if it is definable as a subset of the $\mathcal D^r$ manifold $(M, \{\psi_j\}_{j \in J})$.
	
	The Cartesian product of two $\mathcal D^r$ manifold is naturally defined.
	A map $f:S \rightarrow T$ between $\mathcal D^r$ manifolds is \textit{definable} if its graph is definable in $S \times T$.

	(2)
	A definable subset $Z$ of $X$ is called 
	a \textit{$k$-dimensional $\mathcal D^r$ submanifold}
	of $X$ if each point $x \in Z$ there exist an open box $U_x$ of $x$ in $X$
	and a $\mathcal D^r$ diffeomorphism $\phi_x$ from $U_x$ to some open box $V_x$ of
	$F^d$ such that $\phi_x(x)= 0$ and $U_x \cap Y = \phi_x^{-1}(F^k \cap V_x)$. 
	
	(3) Let $X$ and $Y$ be $\mathcal D^r$ manifolds with $\mathcal D^r$ charts
	$\{\phi_i:U_i \to V_i\}_{i \in A}$ and $\{\psi_j:U'_j \to  V'_j\}_{j \in B}$, respectively. 
	A continuous map $f :X \to Y$ is a \textit{definable $C^r$ map} or a \textit{$\mathcal D^r$ map}
	if for any $i \in A$ and $j \in B$, the image $\phi_i(f^{-1}(V_j') \cap U_i)$ is definable and open in $F^n$ and 
	the map $\psi_j \circ f \circ \phi_i^{-1}:\phi_i(f^{-1}(V_j) \cap U_i) \to F^m$
	is a $\mathcal D^r$ map.
	
	(4) Let $X$ and $Y$ be $\mathcal D^r$ manifolds.
	We say that $X$ is \textit{definably $C^r$ diffeomorphic to} $Y$ or $\mathcal D^r$ diffeomorphic to $Y$
	if there exist $\mathcal D^r$ maps
	$f:X \to Y$ and $h: Y \to X$ such that $f \circ h=\operatorname{id}$ and $h \circ f = \operatorname{id}$.
	
	(5) A $\mathcal D^r$ manifold is called \textit{definably normal} if for any definable closed subset $C$ and any definably open subset $U$ of $M$ with $C \subseteq U$, there exists a definable open subset $V$ of $M$ such that $C \subseteq V \subseteq \mycl_M(V) \subseteq U$,
	where $\mycl_M(V)$ denotes the closure of $V$ in $M$.
\end{definition}

\begin{definition}
For a set $X$, a family $\mathcal K$ of subsets of $X$ is called a \textit{filtered collection} if, for any $B_1, B_2 \in \mathcal K$, 
there exists $B_3 \in \mathcal K$ with $B_3 \subseteq B_1 \cap B_2$. 

Consider an expansion of a dense linear order without endpoints $\mathcal F=(F;<,\ldots)$.
Let $X$ and $T$ be $\mathcal{D}^r$ manifolds.
The parameterized family $\{S_t\}_{t \in T}$ of definable subsets of $X$ is called \textit{definable} if the union $\bigcup_{t \in T} \{t\} \times S_t$ is definable in $T \times X$.

A parameterized family $\{S_t\}_{t \in T}$ of definable subsets of $X$ is a \textit{definable filtered collection} if it is simultaneously definable and a filtered collection.

A definable space $X$ is \textit{definably compact} if every definable filtered collection of closed nonempty subsets of $X$ has a nonempty intersection.
This definition is found in \cite[Section 8.4]{J}.
\end{definition}


%
%
%
%

\begin{definition}
	Consider a definably complete expansion of an ordered field whose universe is $F$.
	Let $\pi:F^n \to F^d$ be a coordinate projection.
	A $\mathcal D^r$ submanifold $M$ of $F^n$ of dimension $d$ is called a \textit{$\mathcal D^r$ multi-valued graph (with respect to $\pi$)} if, for any $x \in M$, there exist an open box $U$ in $F^n$ containing the point $x$ and a $\mathcal D^r$ map $\tau:\pi(U) \to F^n$ such that $M \cap U=\tau(\pi(U))$ and $\pi \circ \tau$ is the identity map defined on $\pi(U)$. 
\end{definition}

\begin{lemma}
	Consider a definably complete expansion of an ordered field whose universe is $F$.
	Let $M$ be a $\mathcal D^r$ submanifold $M$ of dimension $d$.
	Let $\Pi_{n,d}$ be the set of coordinate projections from $F^n$ onto $F^d$.
	Let $U_{\pi}$ be the set of points $x$ at which there exist an open box $U$ in $F^n$ containing the point $x$ and a $\mathcal D^r$ map $\tau:\pi(U) \to F^n$ such that $M \cap U=\tau(\pi(U))$ and $\pi \circ \tau$ is the identity map defined on $\pi(U)$. 
	Then $U_{\pi}$ is a $\mathcal D^r$ multi-valued graph with respect to $\pi$ and 
	$\{U_{\pi}\}_{\pi \in \Pi_{n,d}}$ is a definable open cover of $M$.
\end{lemma}
\begin{proof}
	It is obvious that $U_{\pi}$ is a $\mathcal D^r$ multi-valued graph with respect to $\pi$.
	It is also obvious that $U_{\pi}$ is open in $M$.
	The family $\{U_{\pi}\}_{\pi \in \Pi_{n,d}}$ is a definable open cover of $M$ by \cite[Lemma 3.5, Corollary 3.8]{Fuji}.
\end{proof}

\begin{proposition}[Definable Sard]\label{prop:definable_Sard}
	Consider a d-minimal expansion of an ordered field $\mathcal F=(F,<,+,\cdot, 0,1,\ldots)$.
	Let $M$ be a $\mathcal D^1$ submanifold of $F^m$ of dimension $d$ and $f=(f_1,\ldots, f_n):M \to F^n$ be a $\mathcal D^1$ map.
	The set of critical values of $f$ is definable and of dimension smaller than $n$. 
\end{proposition}
\begin{proof}
	The proposition is trivial when $d<n$ by \cite[Lemma 4.5]{F}.
	We consider the case in which $d \geq n$ in the rest of proof. 
	
	We denote the set of critical values of $f$ by $\Sigma_f$.
	The $\mathcal D^1$ manifold $M$ is covered by finitely many $\mathcal D^1$ multi-valued graphs $U_1,\ldots, U_k$ by Lemma 2.6.
	The equality $\Sigma_f=\bigcup_{i=1}^k \Sigma_{f|_{U_i}}$ obviously holds, where $f|_{U_I}$ is the restriction of $f$ to $U_i$.
	The set $\Sigma_f$ is definable if $ \Sigma_{f|_{U_i}}$ is definable for every $1 \leq i\leq k$.
	We have $\dim \Sigma_f=\max \{\dim \Sigma_{f|_{U_i}}\;|\;1 \leq i \leq k\}$ by \cite[Lemma 4.5]{F}.
	Therefore, we may assume that $M$ is a $\mathcal D^1$ multi-valued graph with respect to a coordinate projection $\pi$.
	We may further assume that $\pi$ is the projection onto the first $d$ coordinates by permuting the coordinates if necessary.
	
	By the definition of $\mathcal D^1$ multi-valued graphs, for any $a \in M$, there exists a nonempty open box $B$ such that $M \cap B$ is the graph of $\mathcal D^1$ map defined on $\pi(B)$.
	In particular, the restriction of $\pi$ to $M \cap B$ is a $\mathcal D^1$ diffeomorphism onto $\pi(B)$.
	The global coordinate functions $x_1,\ldots, x_d$ are local coordinates of $M$ at $a$.
	Let $Df:M \to F^{d \times n}$ be the map giving the Jacobian matrix $(\frac{\partial f_i}{\partial x_j})_{1 \leq i \leq n, 1 \leq j \leq d}$.
	Set $\Gamma_f=\{x \in M\;|\; \operatorname{rank}(Df(x)) < n\}$. 
	The set $\Gamma_f$ is definable.
	The set $\Sigma_f$ is also definable because the equality $\Sigma_f=f(\Lambda_f)$ holds.
	
	Assume for contradiction that $\dim \Sigma_f=n$.
	The definable set $\Sigma_f$ contains a nonempty open box $U$.
	We can take a definable map $g:U \to \Sigma_f$ such that $f \circ g$ is an identity map on $U$.
	We may assume that $g$ is of class $\mathcal C^1$ by \cite[Lemma 3.14]{F} by shrinking $U$ if necessary.
	By differentiation, the matrix$Df(g(x)) \cdot Dg(x)$ is the identity matrix of size $n$.
	It implies that $Df(g(x))$ has rank at least $n$, which contradicts the definition of $\Gamma_f$. 
\end{proof}

\begin{lemma}\label{lem:Morse_basic}
	Consider a d-minimal expansion of an ordered field $\mathcal F=(F,<,+,\cdot, 0,1,\ldots)$.
	Let $\pi$ be the coordinate projection of $F^n$ onto the first $d$ coordinates.
	Let $U$ be a $\mathcal D^2$ multi-valued graph with respect to $\pi$ and $f:U \to F$ be a $\mathcal D^2$ map.
	We can find $a_1,\ldots, a_d \in F$ such that the definable function $\Phi:U \to F$ given by  $\Phi(x_1,\ldots,x_n) =f(x_1,\ldots,x_n)-\sum_{i=1}^d a_ix_i$ is a Morse function and $|a_1|,\ldots,|a_d|$ are sufficiently small.
\end{lemma}
\begin{proof}
	Observe that the global coordinate functions $x_1,\ldots,x_d$ in $F^n$ is a local coordinate function of $U$ by the definition of multi-valued graphs as proven in the proof of Proposition \ref{prop:definable_Sard}.
	Consider the map $H:U \to F^d$ given by $H(x)=(\frac{\partial f}{\partial x_1}(x), \ldots, \frac{\partial f}{\partial x_d}(x))$.
	Observe that $p_0 \in U$ is a critical point of $H$ if and only if $\det(H_f)(p_0)=\det(H_{\Phi})(p_0)=0$, where $H_f$ and $H_{\Phi}$ are Hessians of $f$ and $\Phi$, respectively.
	Let $(a_1,\ldots, a_d)$ be the point in $F^d \setminus \Sigma_H$, where $\Sigma_H$ denotes the set of critical values of $H$.
	We can choose such $(a_1,\ldots, a_d)$ so that $|a_1|,\ldots,|a_d|$ are sufficiently small because $\Sigma_H$ has an empty interior by Proposition \ref{prop:definable_Sard}.
	It is easy to check that $\Phi$ is a Morse function.
	We omit the details.
\end{proof}

\begin{lemma}\label{lem:zeroset}
	Consider a d-minimal expansion of an ordered field $\mathcal F=(F,<,+,\cdot, 0,1,\ldots)$.
	Let $M$ be a $\mathcal D^r$ submanifold of $F^m$ with $0 \leq r < \infty$. 
	Given a definable closed subset $X$ of $M$,
	there exists a $D^r$ function $f:M \rightarrow \mathbb R$ whose zero set is $X$.
\end{lemma}
\begin{proof}
	Consider the closure $\mycl(X)$ of $X$ in $F^m$.
	There exists a $\mathcal D^r$ function $G: F^m \rightarrow  F$ with $G^{-1}(0)=\mycl(X)$ by \cite{MT}.
	The restriction of $G$ to $M$ satisfies the requirement.
\end{proof}

\begin{lemma}\label{lem:sep}
	Consider a d-minimal expansion of an ordered field $\mathcal F=(F,<,+,\cdot, 0,1,\ldots)$.
	Let $M$ be a definable $\mathcal D^r$ submanifold of $F^m$ with $0 \leq r < \infty$. 
	Let $X$ and $Y$ be closed definable subsets of $M$ with $X \cap Y = \emptyset$.
	Then, there exists a $\mathcal D^r$ function $f:M \rightarrow [0,1]$ with $f^{-1}(0)=X$ and $f^{-1}(1)=Y$.
\end{lemma}
\begin{proof}
	There exist $\mathcal D^r$ functions $g,h:M \rightarrow F$ with $g^{-1}(0)=X$ and $h^{-1}(0)=Y$ by Lemma \ref{lem:zeroset}. The function $f:M \rightarrow [0,1]$ defined by $f(x)=\frac{g(x)^2}{g(x)^2+h(x)^2}$ satisfies the requirement.
\end{proof}

\begin{lemma}\label{lem:middle}
	Consider a d-minimal expansion of an ordered field $\mathcal F=(F,<,+,\cdot, 0,1,\ldots)$.
	Let $M$ be a $\mathcal D^r$ submanifold with $0 \leq r < \infty$. 
	Let $C$ and $U$ be closed and open definable subsets of $M$, respectively. 
	Assume that $C$ is contained in $U$.
	Then, there exists an open definable subset $V$ of $M$ with $C \subseteq V \subseteq \mycl(V) \subseteq U$.
\end{lemma}
\begin{proof}
	There is a definable continuous function $h:M \rightarrow [0,1]$ with $h^{-1}(0)=C$ and $h^{-1}(1) = M \setminus U$ by Lemma \ref{lem:sep}.
	The set $V=\{x \in M\,;\, h(x)<\frac{1}{2}\}$ satisfies the requirement.
\end{proof}

\begin{lemma}[Fine definable open covering]\label{lem:covering}
	Consider a d-minimal expansion of an ordered field $\mathcal F=(F,<,+,\cdot, 0,1,\ldots)$.
	Let $M$ be a $\mathcal D^r$ submanifold with $0 \leq r < \infty$. 
	Let $\{U_i\}_{i=1}^q$ be a finite definable open covering of $M$.
	For each $1 \leq i \leq q$, 
	there exists a definable open subset $V_i$ of $M$ satisfying the following conditions:
	\begin{itemize}
		\item  the closure $\mycl(V_i)$ in $M$ is contained in $U_i$ for each $1 \leq i \leq q$, and 
		\item  the collection $\{V_i\}_{i=1}^q$ is again a finite definable open covering of $M$. 
	\end{itemize}
\end{lemma}
\begin{proof}
	We inductively construct $V_i$ so that $\mycl(V_i) \subset U_i$ and $\{V_i\}_{i=1}^{k-1} \cup \{U_i\}_{i=k}^q$ is a finite definable open covering of $M$.
	We fix a positive integer $k$ with $k \leq q$.
	Set $C_k = M \setminus (\bigcup_{i=1}^{k-1} V_i \cup \bigcup_{i=k+1}^{q} U_i)$.
	The set $C_k$ is a definable closed subset of $M$ contained in $U_k$.
	There exists a definable open subset $V_k$ of $M$ with $C_k \subseteq V_k \subseteq \mycl(V_k) \subseteq U_k$ by Lemma \ref{lem:middle}.
	It is obvious that $\{V_i\}_{i=1}^{k} \cup \{U_i\}_{i=k+1}^q$ is a finite definable open covering of $M$.
\end{proof}


\begin{definition}
Let $\mathcal F=(F,<,+,\cdot, 0,1,\ldots)$ be a definably complete expansion of an ordered field.
Let $M$ be a $\mathcal D^r$ submanifold of $F^n$ and $\mathcal D^r(M)$ be the set of $\mathcal D^r$ functions.
The space $\mathcal D^r(M)$ equips the topology defined in \cite{E}. 
\end{definition}

\begin{theorem}\label{thm:Morse}
	Consider a d-minimal expansion of an ordered field $\mathcal F=(F,<,+,\cdot, 0,1,\ldots)$.
	Let $M$ be a definably compact $\mathcal D^2$ submanifold of $F^n$.
	The set of all definable Morse functions on $M$ is open and dense in $\mathcal D^2(M)$.
\end{theorem}
\begin{proof}
	We first prepare several notations and define several sets and maps for later use.
	Set $d=\dim M$.
	The $\mathcal D^2$ submanifold  $M$ is covered by finitely many $\mathcal D^2$ multi-valued graphs $U_1,\ldots, U_k$ by Lemma 2.6.
	Let $U_i$ be a $\mathcal D^2$ multi-valued graph with respect to a coordinate projection $\pi_i:F^n \to F^d$ for $1 \leq i \leq k$.
	Let $\{V_i\}_{i=1}^k$ be a fine definable open covering of $\{U_i\}_{i=1}^k$  given in Lemma 2.12.
	Set $C_i=\mycl(V_i)$.
	Observe that $C_i \subseteq U_i$ and $M=\bigcup_{i=1}^k C_i$.
	The definable set $C_i$ is a closed subset of $M$.
	It deduces that $C_i$ is definably compact.
	By Lemma \ref{lem:sep}, we can take a definable function $\lambda_i:M \to [0,1]$ such that $\lambda_i^{-1}(0)=M \setminus U_i$ and $\lambda_i^{-1}(1)=C_i$.
	
	Let $D_i$ be the $\mathcal D^1$ vector field such that $D_i(x)$ is the projection image of $\partial/\partial x_i$ onto the tangent space $T_xM$ of $M$ at $x \in M$.
	We can find a subset $I_i$ of $\{1,\ldots,n\}$ of cardinality $d$ such that $\pi_i$ is the projection onto the $d$ coordinates $(x_j\;|\; j \in I_i)$.
	The $j$-th smallest element in $I_i$ is denoted by $\sigma_i(j)$. 
	By the definition of multi-valued graphs, $T_xM$ is spanned by $(D_j(x)\;|\; j\in I_i)$ for every $x \in U_i$ and for every $1 \leq i \leq k$.
	
	Let $\mathcal U$ be the set of all all definable Morse functions on $M$.
	We first show that $\mathcal U$ is open.
	We prove a stronger claim for later use.
	For every nonempty subset $I$ of $\{1,\ldots, k\}$, we set $$\mathcal U_I:=\{f \in \mathcal D^2(M)\;|\; f \text{ has no degenerate critical points on }C_i \text{ for each } i \in I\}.$$
	We prove that $\mathcal U_I$ is open.
	Take an arbitrary definable Morse function $h:M \to F$.
	We set $h_j^i=D_{\sigma_i(j)}h$ for $1 \leq j \leq d$ and $i \in I$.
	We set $H_i=\det (D_{\sigma_i(j_1)}D_{\sigma_i(j_2)}h)_{1 \leq j_1,j_2 \leq d}$.
	They are definable continuous functions defined on $M$.
	Since $h$ has no degenerate critical points on $C_i$ for $i \in I$ and the coordinates $x_{\sigma_i(1)}, \ldots, x_{\sigma_i(d)}$ give a local coordinate of $M$ at $x \in U_i$, the function $G_i:=\sum_{j=1}^d |h_j^i|+|H_i|$ is positive on $U_i$ for $i \in I$.
	In particular, we can find a positive $K_i \in F$ such that $G_i>K_i$ on the definably compact definable set $C_i$ for $i \in I$.
	We can take $L_i>0$ so that $|D_{\sigma_i(j_1)}D_{\sigma_i(j_2)}h|<L_i$ on $C_i$.
	Take a sufficiently small $\varepsilon>0$ so that $d!((L_i+\varepsilon)^d-L_i^d)+d\varepsilon<K_i$ for every $i \in I$.
	Consider the open set
	\begin{align*}
	\mathcal V_{h,\varepsilon}&=\{g \in \mathcal D^2(M)\;|\;|g-h|<\varepsilon, |D_j(g-h)|<\varepsilon\ (1 \leq j \leq d),\\
	&\quad  |D_{j_1}D_{j_2}(g-h)|<\varepsilon\ (1 \leq j_1,j_2 \leq d)\}
	\end{align*}
	in $\mathcal D^2(M)$.
	We can verify that $\sum_{j=1}^d |D_{\sigma_i(j)}g|+|\det (D_{\sigma_i(j_1)}D_{\sigma_i(j_2)}g)_{1 \leq j_1,j_2 \leq d}|>0$ on $C_i$ for every $i \in I$ and $g \in \mathcal V_{h,\varepsilon}$.
	It deduces that $\mathcal V_{h,\varepsilon} \subseteq \mathcal U_I$ and $\mathcal U_I$ is open.
	
	We next show that $\mathcal U$ is dense in $\mathcal D^2(M)$.
	We first take arbitrary $h \in \mathcal D^2(M)$.
	We define an open set $\mathcal V_{h,\varepsilon}$ for positive $\varepsilon \in F$ in the same manner as above.
	Since $M$ is definably compact, every positive definable continuous function is bounded from below by a positive constant.
	This deduces that $\{\mathcal V_{h,\varepsilon}\}_{\varepsilon>0}$ is a basis of open neighborhoods of $h$ in $\mathcal D^2(M)$. 
	
	Fix an arbitrary positive $\varepsilon \in F$ and set $\varepsilon'=\varepsilon/k$.
	We have only to construct a function $g \in \mathcal U \cap \mathcal V_{h,\varepsilon}$ so as to show that $\mathcal U$ is dense in $\mathcal D^2(M)$.
	Set $\mathcal U_i:=\mathcal U_{\{1,\ldots,i\}}$ for $1 \leq i \leq k$ for simplicity.
	We construct $g_i \in \mathcal U_i \cap \mathcal V_{h,i\varepsilon'}$.
	It is obvious that $g:=g_k \in \mathcal U \cap \mathcal V_{h,\varepsilon}$.
	
	We construct $g_i$ by induction on $i$.
	We may assume that $\pi_i$ is the projection onto the first $d$ coordinates by permuting the coordinates if necessary.
	We first consider the case in which $i=1$.
	We can find $a_1,\ldots, a_d$ such that $|a_j|<\varepsilon'$ for $1 \leq j \leq d$ and $g_1(x)=h(x)+\sum_{i=1}^da_ix_i$ is a Morse function by Lemma \ref{lem:Morse_basic}.
	It is obvious that $g_1 \in \mathcal U_1 \cap \mathcal V_{h,\varepsilon'}$.

	We next consider the case in which $i>1$.
	We can find $g_{i-1} \in  \mathcal U_{i-1} \cap \mathcal V_{h,(i-1)\varepsilon'}$ by induction hypothesis.
	We construct $g_i \in \mathcal U_i \cap \mathcal V_{g_{i-1},\varepsilon'}$.
	Such a $g_i$ obviously belongs to $\mathcal U_i \cap \mathcal V_{h,i\varepsilon'}$.
	We have already shown that $\mathcal U_{i-1}$ is open.
	Therefore, we can find $\delta>0$ such that $\mathcal V_{g_{i-1},\delta} \subseteq \mathcal U_{i-1} \cap  \mathcal V_{g_{i-1},\varepsilon'}$ because $\{\mathcal V_{g_{i-1},\varepsilon''}\}_{\varepsilon''>0}$ is a basis of open neighborhoods of $g_{i-1}$ in $\mathcal D^2(M)$. 
	
	Set $g_i:=g_{i-1}+\lambda_i \cdot (\sum_{i=1}^da_ix_i)$ for $a_1,\ldots,a_d \in F$.
	We want to choose $a_1,\ldots,a_d \in F$ satisfying the following conditions:
	\begin{enumerate}
		\item[(1)] $g'_i:=g_{i-1}+\sum_{i=1}^da_ix_i$ has no degenerate critical points in $U_i$.
		\item[(2)] $g_i \in \mathcal V_{g_{i-1},\delta}$.
	\end{enumerate}
	We check that $g_i$ belong to $U_i \cap \mathcal V_{g_{i-1},\varepsilon'}$ when $a_1,\ldots,a_d $ satisfy the above conditions (1) and (2).
	It is obvious that $g_i \in V_{g_{i-1},\varepsilon'}$ by the inclusion $\mathcal V_{g_{i-1},\delta} \subseteq  \mathcal V_{g_{i-1},\varepsilon'}$.
	The inclusion $\mathcal V_{g_{i-1},\delta} \subseteq \mathcal U_{i-1}$ implies that $g_i$ has no degenerate critical points on $C_j$ for $1 \leq j \leq i-1$.
	Since $\lambda_i$ is identically one on $C_i$, we have $g_i=g'_i$ on $C_i$.
	Condition (1) implies that $g_i$ has no degenerate critical points in $C_i$.
	We have shown that $g_i$ has no degenerate critical points on $C_j$ for $1 \leq j \leq i$, and this means $g_i \in \mathcal U_i$.
	
	The remaining task is to find $a_1,\ldots,a_d \in F$ so that conditions (1) and (2) are satisfied.
	The following inequalities are satisfied:
	\begin{align*}
		|g_i-g_{i-1}| & \leq  \sum_{l=1}^d|a_l||\lambda_ix_l|  < K \cdot \sum_{j=1}^k|a_j|\\
		|D_j(g_i-g_{i-1}) |& \leq  \sum_{l=1}^d|a_l||D_j(\lambda_ix_l)| < K \cdot\sum_{j=1}^k|a_j|\\
		|D_{j_1}D_{j_2}(g_i-g_{i-1}) |& \leq  \sum_{l=1}^d|a_l||D_{j_1}D_{j_2}(\lambda_ix_l)| < K \cdot\sum_{j=1}^k|a_j|
 	\end{align*}
	Finitely many definable continuous functions $|\lambda_ix_l|$, $|D_j(\lambda_ix_l)|$ and $|D_{j_1}D_{j_2}(\lambda_ix_l)| $ defined on $M$ appear in the above calculation.
	Since $M$ is definably compact, we can find $0<K \in F$ such that these functions are bounded above by $K$ in $M$.
	We used this fact in the calculation.
	We can find $(a_1,\ldots, a_d)$ so that $K \cdot\sum_{j=1}^k|a_j|<\delta$ and $g'_i$ has no degenerate critical points in $U_i$ by Lemma \ref{lem:Morse_basic}.
	This $(a_1,\ldots, a_d)$ satisfies conditions (1) and (2).
\end{proof}

\end{document}